\documentclass[11pt]{amsart}
\usepackage{amssymb, latexsym}

\newcommand{\beqa}{\begin{eqnarray*}}
\newcommand{\eeqa}{\end{eqnarray*}}
\newcommand{\beqn}{\begin{eqnarray}}
\newcommand{\eeqn}{\end{eqnarray}}

\newcommand{\iy}{\infty}
\newcommand{\lt}{\left}
\newcommand{\rt}{\right}

\newcommand{\C}{\mathbb C}
\newcommand{\R}{\mathbb R}

\newcommand{\N}{\mathbb N}

\newcommand{\mcH}{\mathcal H}
\newcommand{\mcB}{\mathcal B}
\newcommand{\mcC}{\mathcal C}

\newcommand{\tf}{\tfrac}

\newcommand{\al}{\alpha}

\newcommand{\e}{\varepsilon}

\newcommand{\de}{\delta}
\newcommand{\De}{\Delta}

\newcommand{\Om}{\Omega}

\newcounter{cnt1}
\newcounter{cnt2}
\newcounter{cnt3}
\newcommand{\blr}{\begin{list}{$($\roman{cnt1}$)$}
 {\usecounter{cnt1} \setlength{\topsep}{0pt}
 \setlength{\itemsep}{0pt}}}
\newcommand{\bla}{\begin{list}{$($\alph{cnt2}$)$}
 {\usecounter{cnt2} \setlength{\topsep}{0pt}
 \setlength{\itemsep}{0pt}}}
\newcommand{\bln}{\begin{list}{$($\arabic{cnt3}$)$}
 {\usecounter{cnt3} \setlength{\topsep}{0pt}
 \setlength{\itemsep}{0pt}}}
\newcommand{\el}{\end{list}}
\newtheorem{thm}{Theorem}[section]

\newtheorem{ex}[thm]{Example}

\newtheorem{Def}[thm]{Definition}

\newtheorem{rem}[thm]{Remark}
\newcommand{\Rem}{\begin{rem} \rm}
\newcommand{\bdfn}{\begin{Def} \rm}
\newcommand{\edfn}{\end{Def}}

\newcommand{\ba}{\begin{array}}
\newcommand{\ea}{\end{array}}

\numberwithin{equation}{section}
\date{}
\begin{document}
\title{\bf{The S-basis and M-basis Problems for Separable Banach Spaces}}
\author[Gill]{Tepper L. Gill}
\address[Tepper L. Gill]{ Departments of Electrical \& Computer Engineering and Mathematics Howard University\\
Washington DC 20059 \\ USA, {\it E-mail~:} {\tt tgill@howard.edu}}
\date{}
\begin{abstract}   This note has two objectives.  The first objective is show that, even if a separable Banach space does not have a Schauder basis (S-basis), there always exists Hilbert spaces $\mcH_1$ and $\mcH_2$, such that $\mcH_1$  is a continuous dense embedding in $\mcB$ and $\mcB$ is a continuous dense embedding in $\mcH_2$.  This is the best possible improvement of a theorem due to Mazur (see \cite{BA} and also \cite{PE1}).  The second objective is show how $\mcH_2$  allows us to provide a positive answer to the Marcinkiewicz-basis (M-basis) problem.   
\end{abstract}
\maketitle
\section{Introduction}
\begin{Def}
Let $\mcB$ separable Banach space, with dual space $\mcB^*$.  A sequence $(x_n) \in \mcB$ is called a S-basis for $\mathcal{B}$ if $\lt\|x_n\rt\|_{\mcB}=1$ and, for each $x \in \mathcal{B}$, there is a unique sequence $(a_n)$ of scalars such that 
\[
x={\rm{lim}}_{n\rightarrow \infty}\sum\limits_{k = 1}^n {a_k x_k } =\sum\limits_{k = 1}^\iy {a_k x_k }. 
\] 
\end{Def}
\begin{Def}Let ${\left\langle { \{x_i : i \in \mathbb{N} \} }\right\rangle}$ be the set of all linear combinations of the family of vectors $\{x_i \}$ (linear span).  The family $\left\{ {\left( {{x_i},x_i^*} \right)} \right\}_{i = 1}^\infty  \subset \mcB \times {\mcB^*}$ is called:
\begin{enumerate}
\item A \emph{fundamental} system if $\overline{\left\langle { \{x_i : i \in \mathbb{N} \} }\right\rangle}=\mcB$.
\item A \emph{minimal} system if $x_j \notin \overline{\left\langle { \{x_i : i \in \mathbb{N} \setminus\{j\} \} }\right\rangle}$.
\item A \emph{total} if for each $x\ne 0$ there exists $i\in \mathbb{N}$ such that $x_i^{*}(x)\ne 0$.
\item A \emph{biorthogonal} system if $x_i^*(x_j)=\de_{ij}$, for all $i,j\in \mathbb{N}$.
\item A \emph{M-basis} if it is a fundamental minimal, total and biorthogonal system. 
\end{enumerate}
\end{Def}
The first problem we consider had its beginning with a question raised by Banach.  He asked whether every separable Banach space has a S-basis.   Mazur gave a partial answer. He proved that every infinite-dimensional separable Banach space contains an infinite-dimensional subspace with a S-basis.  

In 1972, Enflo \cite{EN} answered Banach's question in the negative by providing a separable Banach  space  $\mcB$, without a S-basis and without the approximation property (i.e., every compact operator on $\mcB$ is the limit of a sequence of finite rank operators).  Every Banach space with a S-basis has the approximation property and  Grothendieck \cite{GR} proved that if a Banach space had the approximation property, then it would also have a S-basis.  In the first section we show that, given $\mcB$ there exists separable Hilbert spaces $\mcH_1$ and $\mcH_2$ such that  $\mcH_1 \subset \mcB   \subset  \mcH_2$ as continuous dense embeddings. The existence of $\mcH_1$ is the best possible improvement of Mazur's Theorem, while the existence of $\mcH_2$ shows that $\mcB$ is very close to the best possible case in a well-defined manner.

The second problem we consider is associated with a weaker structure discovered by Marcinkiewicz \cite{M}. He showed that every separable Banach space  $\mcB$ has a biorthogonal system  $\{x_n, x_n^*\}$,   with $\overline{\left\langle {\{x_n\} }\right\rangle}=\mcB$. This system has  many of the properties of an S-basis and is now known as a M-basis for $\mcB$.   A well-known open problem for the M-basis is whether one can choose the system $\{x_n, x_n^*\}$ such that $\left\| {x_n } \right\|\left\| {x_n^* } \right\| = 1$ (see Diestel \cite{D}). This  is called the M-basis problem for separable Banach spaces.  It  has been studied by Singer \cite{SI}, Davis and Johnson \cite{DJ}, Ovsepian and Pelczy\'{n}iski \cite{OP},  Pelczy\'{n}iski \cite{PE} and Plichko \cite{PL}.  The work of Ovsepian and Pelczy\'{n}iski \cite{OP} led to the construction of a bounded M-basis, while  that of Pelczy\'{n}iski \cite{PE} and Plichko \cite{PL} led to independent proofs that, for every  $\e>0$, it is possible to find a biorthogonal system with the property that $\left\| {x_n } \right\|\left\| {x_n^* } \right\| < 1+\e$.   The question of whether we can set $\e=0$ has remained unanswered since 1976.  In this case, we provide a positive answer by constructing a biorthogonal system with the property that $\left\| {x_n } \right\|\left\| {x_n^* } \right\| = 1$.
\section{The S-basis Problem}
In  this section, we construct our Hilbert space  rigging of any given separable Banach space as continuous dense embeddings.  We begin with the construction of $\mcH_2$.
\begin{thm} Suppose ${\mathcal{B}}$ is a separable  Banach space, then there exist a separable Hilbert space ${\mathcal{H}}_2$ such that, ${\mathcal{B}} \subset {\mathcal{H}_2}$ as a continuous dense embedding.
\end{thm}
\begin{proof} Let $\{ x_n \} $ be a countable dense sequence in ${\mathcal{B}}$ and let $\{ x_n^* \} $ be any fixed set of corresponding duality mappings (i.e., $x_n^*  \in {\mathcal{B}^*}$, the dual space of ${\mathcal{B}}$ and $
x_n^* (x_n ) = \left\langle {x_n ,x_n^* } \right\rangle  = \left\| {x_n } \right\|_{\mathcal{B}}^2 = \left\| {x_n^* } \right\|_{\mathcal{B}^*}^2 $).   For each $n$, let $t_n =\tf{1}{\lt\|x_n^*\rt\|^{2}2^n}$ and define $\left( {u,v} \right)$ by:
\[
\left( {u,v} \right) = \sum\nolimits_{n = 1}^\infty  {t_n x_n^*(u)} \bar x_n^* (v)=\sum\nolimits_{n = 1}^\infty  {\tf{1}{\lt\|x_n^*\rt\|^2 2^{n}} x_n^* (u)} \bar x_n^*(v).
\]
It is easy to see that $\left( {u,v} \right)$ is an inner product on ${\mathcal{B}}$.  Let $
{\mathcal{H}_2}$ be the completion of ${\mathcal{B}}$ with respect to this inner product.   It is clear that ${\mathcal{B}}$ is dense in ${\mathcal{H}}_2$, and 
\[
\left\| u \right\|_{{\mcH}_2}^2  = \sum\nolimits_{n = 1}^\infty  {t_n \left| {x_n^* (u)} \right|^2 }  \le \sup _n \tf{1}{\lt\|x_n^*\rt\|^2} \left| {x_n^* (u)} \right|^2 = \left\| u \right\|_{\mathcal{B}}^2,
\]
so the embedding is continuous.
\end{proof}
In order to construct our second Hilbert space, we need the following result by Lax \cite{L}.
\begin{thm}[Lax]\label{L: lax} Let $A \in L[{\mathcal{B}}]$.  If $A$ is selfadjoint on ${\mathcal{H}}_2$ (i.e., $\left( {Ax,y} \right)_{\mathcal{H}_2}  = \left( {x,Ay} \right)_{\mathcal{H}_2}, 
\forall x{\text{,}}y \in {\mathcal{B}}$), then  $A$ has a bounded extension to ${\mathcal{H}}_2$ and $
\left\| A \right\|_{\mathcal{H}_2}  \leqslant M \left\| A \right\|_{\mathcal{B}}$ for some positive constant $M$.
\end{thm}
\begin{proof} Let $x  \in {\mathcal{B}}$
 and, without loss, we can assume that $M = 1$ and $\left\| x  \right\|_{\mathcal{H}_2}  = 1$.  Since $A$ is selfadjoint, 
\[
\left\| {Ax } \right\|_{\mathcal{H}_2}^2  = \left( {Ax ,Ax } \right) = \left( {x ,A^2 x } \right) \leqslant \left\| x  \right\|_{\mathcal{H}_2} \left\| {A^2 x } \right\|_{\mathcal{H}_2}  = \left\| {A^2 x } \right\|_{\mathcal{H}_2}. 
\]
Thus, we have $\left\| {Ax } \right\|_{\mathcal{H}_2}^4  \leqslant \left\| {A^4 x } \right\|_{\mathcal{H}_2}$, so it is easy to see that $
\left\| {Ax } \right\|_{\mathcal{H}_2}^{2n}  \leqslant \left\| {A^{2n} x } \right\|_{\mathcal{H}_2}$ for all $n$.  It follows that: 
\[
\begin{gathered}
  \left\| {Ax } \right\|_{\mathcal{H}_2}  \leqslant (\left\| {A^{2n} x } \right\|_{\mathcal{H}_2} )^{1/2n}  \leqslant (\left\| {A^{2n} x } \right\|_\mathcal{B} )^{1/2n}  \hfill \\
  {\text{          }} \leqslant (\left\| {A^{2n} } \right\|_\mathcal{B} )^{1/2n} (\left\| x  \right\|_\mathcal{B} )^{1/2n}  \leqslant \left\| A \right\|_\mathcal{B} (\left\| x  \right\|_\mathcal{B} )^{1/2n}.  \hfill \\ 
\end{gathered} 
\]
Letting $n \to \infty $, we get that $\left\| {Ax } \right\|_{\mathcal{H}_2}  \leqslant \left\| A \right\|_{\mathcal{B}} $ for any $x$ in the dense set of the unit ball  $ B_{\mathcal{H}_2} \cap \mcB$.  Since the norm is attained on a dense set of the unit ball, we are done.
\end{proof} 
For our second Hilbert space, fix $\mcB$ and define ${\mcH}_1$ by:
\[
\begin{gathered}
  {\mcH}_1  = \left\{ {u \in \mcB \left| {\;\sum\nolimits_{n = 1}^\infty  { t_n^{-1}\left| {\left( {u, {x}_n } \right)_{{{}}2 } } \right|^2  < \infty } } \right.} \right\},\quad {\text{with}} \hfill \\
  \quad \quad \quad \quad \left( {u,v} \right)_1  = \sum\nolimits_{n = 1}^\infty  {t_n^{-1} \left( {u, {x}_n } \right)_{2 } \left( {{x}_n ,v } \right)_{{{}}2 } .}  \hfill \\ 
\end{gathered}
\]
For $u \in \mcB$, let ${T}_{12}u$ be defined by ${T}_{12}u = \sum\nolimits_{n = 1}^\infty { t_n\left( {u, {x}_n } \right)_{2} {x}_n}$.
\begin{thm}    The operator ${{T}}_{12}$ is a positive trace class operator on $\mcB$ with a bounded extension to  ${\mathcal{H}}_2$.  In addition, ${\mathcal{H}}_1  \subset {\mathcal{B}} \subset {\mathcal{H}}_2 $ (as continuous dense embeddings),  $\left( {T_{12}^{1/2} u,\;T_{12}^{1/2} v} \right)_1  = \left( {u,\;v} \right)_2$ and  $\left( {T_{12}^{ - 1/2} u,\;T_{12}^{ - 1/2} v} \right)_2  = \left( {u,\;v} \right)_1$. 
\end{thm}
\begin{proof}
First, since terms of the form $\{u_N  = \sum\nolimits_{k = 1}^N { t_k^{-1} } \left( {u,{x}_k } \right)_2 {x}_k :\;u \in \mcB \}$ are dense in ${\mathcal{B}}$, we see that ${\mathcal{H}}_1 $ is dense in ${\mathcal{B}}$.  It follows that ${\mathcal{H}}_1$ is also dense in ${\mathcal{H}}_2 $.

For the operator ${T}_{12}$, we see that  ${\mathcal{B}} \subset {\mathcal{H}}_2 \Rightarrow   \left( { u ,{x}_n } \right)_2$ is defined for all $u \in {\mathcal{B}}$, so that  ${{T}}_{12}$ maps ${\mathcal{B}} \to {\mathcal{B}}$ and:
\[
\left\| {{{T}}_{12} u} \right\|_{\mathcal{B}}^2  \le \left[ {\sum\nolimits_{n = 1}^\infty   {t_n^2 }\left\| {x}_n  \right\|_{\mathcal{B}}^2 } \right]\left[ {\sum\nolimits_{n = 1}^\infty  {\left| {\left( {u,{x}_n } \right)_2 } \right|^2 } } \right] = M\left\| u \right\|_2^2  \le M\left\| u \right\|_{\mathcal{B}}^2.
\]
Thus, ${{T}}_{12}$ is a bounded operator on ${\mathcal{B}}$.  It is clearly trace class and, since $\left( {T_{12} u,\; u} \right)_2= \sum\nolimits_{n = 1}^\infty  {t_n\lt|\left( {u, {x}_n } \right)_{2}\rt|^2}>0$, it is positive.  From here, it's easy to see that $ T_{12}$ is selfadjoint on ${\mathcal{H}}_2$ so, by Theorem 2.2 it has a bounded extension to ${\mathcal{H}}_2$. 

An easy calculation now shows that $\left( {T_{12}^{1/2} u,\;T_{12}^{1/2} v} \right)_1  = \left( {u,\;v} \right)_2$ and  $\left( {T_{12}^{ - 1/2} u,\;T_{12}^{ - 1/2} v} \right)_2  = \left( {u,\;v} \right)_1$. 
\end{proof}
Since the counter example of Enflo, the only direct information about a Banach space without a basis has been the following theorem of Mazur:
\begin{thm}Every infinite dimensional separable Banach contains a infinite dimensional subspace with a basis. 
\end{thm}
Theorems 2.1 and 2.3 show that, even if a Banach space does not have a basis, it is very close to the best possible case. 
\begin{rem}
Historically, Gross \cite{G} first proved that every real separable Banach space $\mcB$ contains a separable Hilbert space (version of $\mcH_1$), as a dense embedding, and that this space is the support of a Gaussian measure.   Then Kuelbs \cite{KB} showed that one can construct $\mcH_2$ so that $\mcH_1 \subset \mcB \subset \mcH_2$ as continuous dense embeddings, with $\mcH_1$  and $\mcH_2$ related by Theorem 2.3.

A particular Gross-Kuelbs construction of $\mcH_2$ was used in \cite{GZ} to provide the foundations for the Feynman path integral formulation of quantum mechanics \cite{FH} (see also \cite{GZ1}). 

This construction was also used in  \cite{GBZS} to show that  every bounded linear operator $A$ on a separable Banach space $\mcB$ has a adjoint $A^*$ defined on $\mcB$, such that (see below): 
\begin{enumerate}
\item $A^ * A$ is m-accretive (i.e., if $x \in \mcB$ and  $x^*$ is a corresponding duality, then $\left\langle {A^ * Ax,{x^*}} \right\rangle \ge 0$),		
\item $(A^ * A)^ *   = A^ * A$ (selfadjoint), and 
\item $I + A^ * A$ has a bounded inverse. 
\end{enumerate}
\end{rem}
\begin{ex}
The following example shows how easy it is to construct an adjoint $A^*$ satisfying all the above conditions, using only $\mcH_2$. Let $\Om$ be a bounded open domain in $\R^n$ with a class $\C^1$ boundary and let $\mcH_0^1[\Om]$, be the set of all real-valued functions $u \in L^2[\Om]$ such that their first order weak partial derivatives are in $L^2[\Om]$ and vanish on the boundary.  It follows that 
\[
\left( {u,v} \right) = \int_\Omega  {\nabla u({\mathbf{x}}) \cdot \nabla v({\mathbf{x}})d{\mathbf{x}}} =\left\langle {u,J_0 v} \right\rangle, 
\]
defines an inner product on $\mcH_0^1[\Om]$, where $J_0$ is the conjugate isomorphism between $\mcH_0^1[\Om]$ and its dual ${\mcH^{ - 1}}[\Omega ]$.  The  space  ${\mcH^{ - 1}}[\Omega ]$ coincides with the set of all distributions of the form
\[
u = {h_0} + \sum\limits_{i = 1}^n {\frac{{\partial {h_i}}}{{\partial {x_i}}}} ,\quad {\text{where}}\;{h_i} \in {L^2}[\Omega ],\quad 1 \leqslant i \leqslant n.
\] 
In this case we also have for $p \in [2, \iy)$ and $q \in (1, 2], \tf{1}{p}+\tf{1}{q}=1$ that,  
\[
\mcH_0^1[\Omega ] \subset {L^p}[\Omega ] \subset {L^q}[\Omega ] \subset {\mcH^{ - 1}}[\Omega ]
\]  
all as continuous dense embeddings.

From the inner product on $\mcH_0^1[\Om]$ we see that $J_0=-\De$, the Laplace operator under Dirichlet homogeneous boundary conditions on $\Om$.  If we set $\mcH_1=\mcH_0^1[\Om], \; \mcH_2=\mcH^{-1}$ and $J=J_0^{-1}$, then for every $A \in \mcC[L^p(\Om)]$ (i.e., the closed densely defined linear operators on $L^p(\Om)$), we  obtain $A^* \in \mcC[L^p(\Om)]$,  where $A^*=J^{-1}A'J |_p =[-\De] A'[-\De]^{-1}|_p$ for each $A' \in \mcC[L^{q }(\Om)]$.  It is now easy to show that $A^*$ satisfies the conditions (1)-(3) above for an adjoint operator on $L^p(\Om)$.
\end{ex}
\section{The M-basis Problem}
To understand the M-basis problem and its solution in a well-known setting, let  $\R^2$ have its standard inner product $( \cdot ,  \cdot )$ and let $x_1, \; x_2$ be any two
independent basis vectors. Define a new inner product on $\R^2$ by
\beqn
\begin{gathered}
  \left\langle {y}
 \mathrel{\left | {\vphantom {y z}}
 \right. \kern-\nulldelimiterspace}
 {z} \right\rangle  = {t_1}\left( {x_1^{} \otimes x_1^{}} \right)\left( {y \otimes z} \right) + {t_2}\left( {x_2^{} \otimes x_2^{}} \right)\left( {y \otimes z} \right) \hfill \\
  \quad \quad  = {t_1}\left( {y,x_1^{}} \right)\left( {z,x_1^{}} \right) + {t_2}\left( {y,x_2^{}} \right) \left( {z,x_2^{}} \right), \hfill \\ 
\end{gathered} 
\eeqn
where $t_1, \, t_2>0, \; t_1+t_2=1$.  Define new functionals $S_1$ and $S_2$ by:

\[
{S_1}(x) = \frac{{\left\langle {x}
 \mathrel{\left | {\vphantom {x {{x_1}}}}
 \right. \kern-\nulldelimiterspace}
 {{{x_1}}} \right\rangle }}{{{\alpha _1}\left\langle {{{x_1}}}
 \mathrel{\left | {\vphantom {{{x_1}} {{x_1}}}}
 \right. \kern-\nulldelimiterspace}
 {{{x_1}}} \right\rangle }},\quad {S_2}(x) = \frac{{\left\langle {x}
 \mathrel{\left | {\vphantom {x {{x_2}}}}
 \right. \kern-\nulldelimiterspace}
 {{{x_2}}} \right\rangle }}{{{\alpha _2}\left\langle {{{x_2}}}
 \mathrel{\left | {\vphantom {{{x_2}} {{x_2}}}}
 \right. \kern-\nulldelimiterspace}
 {{{x_2}}} \right\rangle }},\quad {\text{for}}\quad y \in {\mathbb{R}^2}.
\]
Where $\al_1, \; \al_2 >0$ are chosen to ensure that $\left\| {{S_1}} \right\| = \left\| {{S_2}} \right\| = 1$.
Note that, if $(x_1, x_2)=0$, $S_1$ and $S_2$ reduce to 
\[
{S_1}(x) = \frac{{\left( {x,x_1^{}} \right)}}{{{\alpha _1}\left\| {{x_1}} \right\|}},\quad {S_2}(x) = \frac{{\left( {x,x_2^{}} \right)}}{{{\alpha _2}\left\| {{x_2}} \right\|}}.
\]
Thus, we can define many equivalent inner products on $\R^2$ and many linear functionals with the same properties but different norms. 

The following example shows how this construction can be of use. 
\begin{ex}
In this example, let $x_1=e_1$ and $x_2=e_1+e_2$, where $e_1=(1,0), \; e_2=(0,1)$. In this case, the biorthogonal functionals are generated by the vectors $\bar{x}_1=e_1-e_2$ and $\bar{x}_2=e_2$ {\rm{(}{i.e.}, $x_1^*(x) = \left( {x,{{\bar x}_1}} \right),\quad x_2^*(x) = \left( {x,{{\bar x}_2}} \right)$\rm{)}}.  It follows that $(x_1,\bar{x}_2)=0, \; (x_1,\bar{x}_1)=1$ and $(x_2, \bar{x}_1)=0, \; (x_2, \bar{x}_2)=1$.  However,  $\left\| {x_1 } \right\|\left\| {\bar{x}_1 } \right\| = \sqrt 2 ,\quad \left\| {x_2 } \right\|\left\| {\bar{x}_2 } \right\| = \sqrt 2$, so that 
$\left\{ {{x_1},\left( {\; \cdot ,\;{{\bar x}_1}} \right)} \right\}$ and $\left\{ {{x_2},\left( {\; \cdot ,\;{{\bar x}_2}} \right)} \right\}$ fails to solve the M-basis problem on $\R^2$.
 
In this case, we  set $\al_1=1$ and $\al_2= \left\| {{{x}_2}} \right\|$ so that, without changing $x_1$ and $x_2$, and using the inner product from equation (1.1) in the form 
\[
  \left\langle {x}
 \mathrel{\left | {\vphantom {x y}}
 \right. \kern-\nulldelimiterspace}
 {y} \right\rangle  = {t_1}\left( {x,\bar{x}_1^{}} \right)\left( {y,\bar{x}_1^{}} \right) + {t_2}\left( {x,\bar{x}_2^{}} \right) \left( {y,\bar{x}_2^{}} \right), 
\]
$S_1$ and $S_2$ become
\[
{S_1}(x) = \frac{{\left( {x,\bar{x}_1^{}} \right)}}{{\left\| {{\bar{x}_1}} \right\|}},\quad {S_2}(x) = \frac{{\left( {x,\bar{x}_2^{}} \right)}}{{\left\| {{{x}_2}} \right\|}}.
\]
It now follows that $S_i(x_i)=1$ and $S_i(x_j)=0$ for $i \ne j$ and    $\left\| {S_i^{} } \right\|\left\| {x_i^{} } \right\| = 1$,  so that system $\{ {x_1, S_1 } \}$ and $\{ {x_2, S_2 } \}$  solves the M-basis problem.
\end{ex}
\begin{rem} For a given set of independent vectors on a finite dimensional vector space, It is known that the corresponding biorthogonal functionals are unique.  This example shows that uniqueness is only up to a scale factor and this is what we need to produce an M-basis. 
\end{rem}

The following theorem shows how our solution to the M-basis problem for $\R^2$ can be extended to any separable Banach space.

\begin{thm} Let $\mcB$ be a infinite-dimensional separable Banach space.  Then $\mcB$ contains  an  M-basis with the property that $\left\| {x_i } \right\|_{\mcB} \left\| {x_i^* } \right\|_{{\mcB}^*} = 1$ for all $i$. 
\end{thm}
\begin{proof}
Construct $\mcH=\mcH_2$ via Theorem 2.1, so that  $\mcB \subset \mcH$ is a dense continuous embedding and let ${\left\{ {x_i } \right\}_{i=1}^\iy}$ be a fundamental minimal system for $\mcB$.   If $i \in \N$, let $M_{i, \mcH}$ be the closure of the span of ${\left\{ {x_i } \right\}}$ in $\mcH$.  Thus, $x_i \notin M_{i, \mcH}^\bot$, $M_{i,\mcH}  \oplus M_{i,\mcH}^ \bot   = \mcH$ and $\left( {y,x_i } \right)_\mcH = 0$ for all $y \in M_{i, \mcH}^\bot$.  

Let $\hat{M}_{i}$ be the closure of the span of ${\left\{ {x_j } \; j\ne i \right\}}$ in $\mcB$.  Since $\hat{M}_{i} \subset M_{i, \mcH}^\bot$ and $x_i \notin \hat{M}_{i},  \; (y, x_i)_\mcH=0$ for all $y \in \hat{M}_{i}$.  Let the seminorm $p_i(\, \cdot \,)$ be defined on the closure of the span of $\{x_i \}$, in $\mcB$  by $p_i(y) = \left\| {x_i } \right\|_\mcB \left\| y \right\|_\mcB$,
and define $ {\hat x}_i^* (\, \cdot \,)$ by:
\[
 {\hat x}_i^* (y) = \frac{{\left\| x_i \right\|_{\mcB}^2 }}
{{\left\| x_i \right\|_{\mcH}^2 }}\left( {y, x_i} \right)_{\mcH} 
\]
By the Hahn-Banach Theorem, ${\hat x}_i^* (\, \cdot \,)$ has an extension $x_i^* (\, \cdot \,)$ to $\mcB$, such that $\left| {x_i^* (y)} \right|   \leqslant p_i (y)= \left\| x_i \right\|_B \left\| y \right\|_B$ for all $y \in \mcB$. By definition of $p_i(\, \cdot \,)$, we see that $\left\| {x_i^* } \right\|_{\mathcal{B}^*} \le \left\| x_i \right\|_\mathcal{B}$.   On the other hand  $x_i^* (x_i) = \left\| x_i \right\|_\mathcal{B}^2 \leqslant \left\| x_i \right\|_\mathcal{B} \left\| {x_i^* } \right\|_{\mathcal{B}^*}$, 
so that $x_i^* ( \, \cdot \, )$ is a duality mapping for $x_i$.  If $x_i^*(x)=0$ for all $i$, then $x \in \bigcap_{i=1}^\iy{\hat{M}_i} =\{0\}$ so that the family ${\left\{ {x_i^* } \right\}_{i=1}^\iy}$ is total.  If  we let $\left\| x_i \right\|_\mathcal{B}=1$, it is clear that $x_i^*(x_j)= \de_{ij}$, for all $i,j \in \N$. Thus, $\{x_i, x_i^* \}$ is an  M-basis system with $\left\| {x_i } \right\|_{\mcB} \left\| {x_i^* } \right\|_{{\mcB}^*} = 1$ for all $i$.
\end{proof}
\section*{Conclusion}
In this paper we have first shown that every infinite dimensional separable Banach space is very close to a Hilbert space in a well defined manner, providing the best possible improvement on  the well-known theorem of Mazur.  We have then provided a solution to the M-basis problem by showing that every infinite dimensional separable Banach space has a M-basis $\{x_i, x_i^* \}$, with the property that $\left\| {x_i } \right\|_{\mcB} \left\| {x_i^* } \right\|_{{\mcB}^*} = 1$ for all $i$.

\end{document}